\documentclass[12pt]{article}

\usepackage{amsmath,amsthm,amsfonts,latexsym,amsopn,verbatim,amscd,amssymb,}
\usepackage{amssymb}
\usepackage{amsfonts}
\usepackage{hyperref}
\usepackage{graphics,color}
\usepackage{graphicx}

\theoremstyle{plain}
\newtheorem{Thm}{Theorem}[section]
\newtheorem{Lem}[Thm]{Lemma}
\newtheorem{Prop}[Thm]{Proposition}
\newtheorem{Cor}[Thm]{Corollary}
\theoremstyle{definition}

\newtheorem{Def}[Thm]{Definition}
\newtheorem{example}[Thm]{Example}

\numberwithin{equation}{section}

\newcommand{\bnum}{\begin{enumerate}}
\newcommand{\enum}{\end{enumerate}}
\begin{document}
\begin{center}
\textbf{Nil Clean Ideal}\\
\end{center}

\begin{center}
Ajay Sharma and Dhiren Kumar Basnet\\
\small{\it Department of Mathematical Sciences, Tezpur University,
 \\ Napaam, Tezpur-784028, Assam, India.\\
Email: ajay123@tezu.ernet.in \\
 dbasnet@tezu.ernet.in}
\end{center}
\noindent \textit{\small{\textbf{Abstract:}  }} Motivated by the concept of clean ideals, we introduce the notion of nil clean ideals of a ring. We define an ideal $I$ of a ring $R$ to be nil clean ideal if every element of $I$ can be written as a sum of an idempotent and a nilpotent element of $R$. In this article we discuss various properties of nil clean ideals.

\bigskip

\smallskip

\noindent \small{\textbf{\textit{$2010$ Mathematics Subject Classification:}}  16N40, 16U99.} \\
\smallskip

\bigskip

\section{INTRODUCTION}
$\mbox{\hspace{.5cm}}$
Here rings $R$ are associative rings with unity unless otherwise indicated. The Jacobson radical, set of units, set of nilpotent elements, set of idempotents and centre of a ring $R$ are denoted by $J(R)$, $U(R)$, $Nil(R)$, $Idem(R)$ and $C(R)$ respectively.
Nicholson\cite{nicholson1977lifting} called an element $x$ of a ring $R$, a clean element, if $x=e+u$ for some $e\in Idem(R)$, $u\in U(R)$ and called $R$, a clean ring if all its elements are clean. H. Chen and M. Chen\cite{chen2003clean}, introduced the concept of clean ideals as follows: an ideal $I$ of a ring $R$ is called clean ideal if for any $x\in I$, $x=u+e$, for some $u\in U(R)$ and $e\in Idem(R)$. Motivated by these ideas we define an ideal $I$ of a ring $R$ to be a nil clean ideal if for any $x\in I$, $x=n+e$, where $n\in Nil(R)$ and  $e\in Idem(R)$ and ideal $I$ of a ring $R$ is called uniquely nil clean ideal if for each $a\in I$, there exists a unique $e\in Idem(R)$ such that $a-e\in Nil(R)$. In this article, we discuss some interesting properties of nil clean ideals of a ring.
\section{Nil clean ideal}
\begin{Def}
  An Ideal $I$ of a ring $R$ is said to be a nil clean ideal if every element of $I$ can be written as a sum of an idempotent and a nilpotent element of $R$.
\end{Def}
Clearly every ideal of a nil clean ring is nil clean ideal. But there exist rings which are not nil clean rings but contains some nil clean ideal, for example the ring $\mathbb{Z}_{p^n}$, where $n>1$ and $p$ is an odd prime, is not nil clean ring but every proper ideal of $\mathbb{Z}_{p^n}$ is nil clean ideal. Another such example is given below.

\begin{example}
  Let $R_1, R_2$ be two rings such that $R_1$ is  nil clean but $R_2$ is not. Consider $R=R_1\oplus R_2$, then $R$ is not nil clean, but the ideal $I=R_1\oplus 0$ is a nil clean ideal of $R$.
\end{example}
\begin{Lem}\label{L1}
  Every nil clean ideal of a ring $R$ is a clean ideal of $R$.
\end{Lem}
\begin{proof}
  Let $I$ be a nil clean ideal of $R$. For $x\in I$, $-x=e+n$, where $e\in Idem(R)$ and $n\in Nil(R)$. Then $x=(1-e)+(-1-n)$, where $1-e \in Idem(R)$ and $-(1+n)\in U(R)$.
\end{proof}
The converse of Lemma \ref{L1} is not true because $I=\{0,2,4\}$ is a clean ideal of $\mathbb{Z}_6$ but not a nil clean ideal of $\mathbb{Z}_6$.
\begin{Prop}\label{PPP1}
If $I$ be a nil clean ideal of a ring $R$ then $I\cap J(R)$ is a nil ideal of $R$.
\end{Prop}
\begin{proof}
  Let $x\in I\cap J(R)$, so $x=e+n$, where $e\in Idem(R)$ and $n\in Nil(R)$, so there exists $k\in \mathbb{N}$ such that $n^k=0$. Now $n^k=(x-e)^k=\sum_{t,r\in R}^{finite}txr +(-1)^ke^k=0\Rightarrow e=(-1)^k\sum_{t,r\in R}^{finite}txr\in J(R)$ as $x\in J(R)$. So $1-e\in Idem(R)\cap U(R)=\{1\}$, hence $1-e=1\Rightarrow e=0\Rightarrow x=n$. Thus the result follows.
\end{proof}
Note that Proposition $3.1.6$\cite{Diesl} follows from Proposition \ref{PPP1}.
\begin{Cor}
  If $R$ is a nil clean ring then $J(R)\subseteq N(R)$. In particular $J(R)$ is a nil ideal.
\end{Cor}
\begin{Prop}
  Product of two nil clean ideals of a commutative ring is again a nil clean ideal.
\end{Prop}
\begin{proof}
  It follows from the fact that sum of nilpotent elements in a commutative ring is again a nilpotent element.
\end{proof}
We define strongly nil clean ideal and uniquely strongly nil clean ideal given below:
\begin{Def}
    An Ideal $I$ of a ring $R$ is said to be strongly nil clean ideal if every element of $I$ can be written as a sum of an idempotent and a nilpotent of $R$ that commutes. $I$ is called uniquely strongly nil clean ideal if the strongly nil clean decomposition of every element of $I$ is unique.
 \end{Def}
\begin{Thm}
   An ideal $I$ of a ring $R$ is strongly nil clean ideal \textit{if and only if} it is strongly clean ideal and $a-a^2$ is nilpotent for all $a$ in $I$.
 \end{Thm}
 \begin{proof}
   Let $I$ be a strongly nil clean ideal of $R$ and $a\in I$. So $a=e+n$, where $e\in Idem(R)$, $n\in Nil(R)$ and $en=ne$, then $a=(1-e)+(2e-1+n)$ is a strongly clean decomposition of $a$. Also $a-a^2=(1-2e-n)n$ is nilpotent. Conversely, let $a\in I$, then $a=e+u$, where $e\in Idem(R)$, $u\in U(R)$ and $eu=ue$. By assumption $a-a^2$ is nilpotent, which implies $(1-2e-u)$ is nilpotent. So $a=(1-e)+(-1+2e+u)$ is a strongly nil clean decomposition of $R$.
 \end{proof}

 \begin{Thm}
   If an ideal $I$ of a ring $R$ is strongly nil clean ideal then
\begin{enumerate}
\item     $I$ is uniquely strongly nil clean ideal of $R$.
\item     $I$ is uniquely strongly clean ideal of $R$.
\end{enumerate}
 \end{Thm}
 \begin{proof}
The proof follows from Lemma 2.3 \cite{Tamer}.
 \end{proof}
In Theorem \ref{TTT1}, we characterise nil clean ideals of a commutative ring which contains the Jacobson radical.
 \begin{Thm}\label{TTT1}
   Let $I$ be an ideal of a commutative ring $R$ and $J(R)\subseteq I$. Then $I/J(R)$ is boolean and $J(R)$ is nil \textit{if and only if} $I$ is nil clean ideal of $R$.
 \end{Thm}
 \begin{proof}
 ($\Rightarrow$)  Let $a\in I$, then clearly $a-a^2\in J(R)$. As $J(R)$ is nil so there exists a polynomial $f(t)$ over $\mathbb{Z}$ such that $e:=f(a)$ is an idempotent of $R$ and $a-e\in J(R)$. Hence $a=e+(a-e)$ is a nil clean decomposition. \par
  ($\Leftarrow$) Clearly $J(R)$ is nil ideal by Proposition \ref{PPP1}. Enough to show $I/J(R)$ is boolean. Let $x\in I/J(R) $, so $x=i+J(R)=e+n+J(R)$, where $e\in Idem(R)$ and $n\in Nil(R)$. As J(R) is nil so $x=e+J(R)$ and hence $x^2=x$, $i.e.$ $I/J(R)$ is boolean.
 \end{proof}
\begin{Lem}
  Every idempotent in a uniquely nil clean ideal is a central idempotent.
\end{Lem}
\begin{proof}
  Let $I$ be a uniquely nil clean ideal of a ring $R$ and $e$ be any idempotent in $I$. For any $x\in R$, since $e=(e-ex(1-e))+ex(1-e)=e+0$, so $(e-ex(1-e))=e\Rightarrow ex=exe$ as $e+ex(1-e)\in Idem(R)$. Similarly we can show that $xe=exe$. Hence $xe=ex$.
\end{proof}

The following theorem shows that, for a nil clean expression of an element of a nil clean ideal of a ring $R$, the nilpotent and idempotent elements are actually elements of the ideal.
  \begin{Thm}\label{main1}
   Let $I$ be an ideal of a ring $R$ then $I$ is nil clean ideal \textit{if and only if} for any $x\in I$, $x=n+e$, where $n\in Nil(I)$ and $e\in Idem(I)$.
 \end{Thm}
 \begin{proof}
   Let $I$ be a nil clean ideal of $R$ and $x\in I$. Then there exists $n\in Nil(R)$ and $e\in Idem(R)$ such that $x=e+n$. Now $n^k=0$, for some $k\in \mathbb{N}$. So $(x-n)^k=(-1)^kn^k+\sum_{finite} q_ixp_i$, for some $p_i, q_i\in R$ $\Rightarrow$ $(x-n)^k=\sum_{finite} q_ixp_i\in I$, so $e^k=e\in I$. Hence $n\in Nil(I)$. The other part follows from definition.
 \end{proof}
 \begin{Cor}
   If $R$ is a local ring then every proper nil clean ideal of $R$ is nil ideal. In fact if $R$ has no non trivial idempotents then every proper nil clean ideal of $R$ is nil ideal.
 \end{Cor}
  \begin{Thm}\label{mmm}
    Let $I$ be an ideal of a commutative ring $R$. Then $I$ is nil clean ideal of $R$ \textit{if and only if} $I/I\cap J(R)$ is boolean and $I\cap J(R)$ is nil.
 \end{Thm}
 \begin{proof}
   $(\Leftarrow)$ Proof of this part is similar to that of theorem \ref{TTT1}. \\
   $(\Rightarrow)$ Clearly by Proposition \ref{PPP1}, $I\cap J(R)$ is nil ideal. Let $x\in I/I\cap J(R)$ then $x=e+n+I\cap J(R)$, where $e\in Idem(I)$ and $n\in Nil(I)$, so $x=e+I\cap J(R)$ and $x^2=x$, as required.
  \end{proof}
 In the following Theorem we characterise a nil clean ring with its nil clean ideals.
 \begin{Thm}\label{main}
     $R$ is a nil clean ring \textit{if and only if} there exists a central idempotent $e$ in $R$ such that ideal generated by $e$ and ideal generated by $1-e$ both are nil clean ideals of $R$.
 \end{Thm}
 \begin{proof}
    If $R$ is a nil clean ring then $e=1$ works. Conversely, let $e$ be a central idempotent in $R$ such that ideals $<e>$ and $<1-e>$, generated by $e$ and $1-e$ respectively are nil clean ideals of $R$. Since $R=<e>+<1-e>$, so for $x\in R$, $x=a+b$, where $a\in <e>$ and $b\in <1-e>$. There exist $f_1\in Idem(<e>)$, $f_2\in Idem(<1-e>)$, $n_1\in Nil(<e>)$ and $n_2\in Nil(<1-e>)$ such that $a=f_1+n_1$ and $b=f_2+n_2$. Hence $x=(f_1+f_2)+(n_1+n_2)$ and it is easy to see that $f_1+f_2\in Idem(R)$ and $n_1+n_2\in Nil(R)$, as required.
 \end{proof}
A finite orthogonal set of idempotents $e_1, \cdot\cdot\cdot , e_n$ in a ring $R$ is said to be complete set if $e_1+ \cdot\cdot\cdot +e_n=1$. The above Theorem\ref{main} is also characterised by complete orthogonal set of central idempotents.
 \begin{Thm}
 $R$ is nil clean ring \textit{if and only if} there exists a complete set of central idempotents $e_1, \cdot\cdot\cdot , e_n$ in $R$ such that ideal generated by $e_i$ is nil clean ideal of $R$, for all $i$.
 \end{Thm}
 \begin{proof}
   $(\Rightarrow)$ Taking $e=1$.\\
   $(\Leftarrow)$ Clearly $<e_1>+<e_2>+\cdot\cdot\cdot+<e_n>=R$ so similar to Theorem \ref{main} we can show that $R$ is nil clean ring.
 \end{proof}
 \begin{Prop}
   Let $I$ be an ideal of a ring $R$. Then the following are equivalent:
   \begin{enumerate}
     \item $I$ is nil clean ideal of $R$.
     \item There exists a complete set of central idempotents $e_1, \cdot\cdot\cdot , e_n$ in $R$ such that $e_iIe_i$ is a nil clean ideal of $e_iRe_i$, for all $i$.
   \end{enumerate}
 \end{Prop}
 \begin{proof}
   $(i)\Rightarrow (ii)$ Taking $e=1$.\\
   $(ii)\Rightarrow (i)$ Let $e_1, \cdot\cdot\cdot , e_n$ be a complete set of central idempotents in $R$ such that $e_iIe_i$ is a nil clean ideal of $e_iRe_i$, for all $i$. It is enough to show the result for $n=2$. Clearly $I=e_1Ie_1\oplus e_2Ie_2$ as $e_1$ is central idempotent. Let $x\in I$, so there exists $f_1\in Idem(e_1Ie_1)$, $m_1\in Nil(e_1Ie_1)$, $f_2\in Idem(e_2Ie_2)$ and $m_2\in Nil(e_2Ie_2)$ such that $x=(f_1+f_2)+(m_1+m_2)$, which is a nil clean expression. Hence $I$ is nil clean ideal of $R$.
 \end{proof}
 \begin{Thm}
   Let $R$ be a ring and $I_1$ be an ideal containing a nil ideal $I$. Then $I_1$ is nil clean ideal of $R$ \textit{if and only if} $I_1/I$ is nil clean ideal of $R/I$.
 \end{Thm}
 \begin{proof}
   If $I_1$ is nil clean of $R$ then clearly $I_1/I$ is nil clean ideal of $R/I$. Conversely, let $I_1/I$ be a nil clean ideal of $R/I$ and $x\in I_1$. Then $\overline{x}=\overline{e}+\overline{n}$, where $\overline{e}\in Idem(I_1/I)$ and $\overline{n}\in Nil(I_1/I)$. Since idempotents can be lifted modulo nil ideal, so lift $\overline{e}$ to $e\in I_1$. Then $x-e$ is nilpotent in $I_1$, modulo $I$ and hence $x-e$ is nilpotent in $I_1$.
 \end{proof}
 \begin{Thm}\label{HOMIMAGE}
   Every homomorphic image of nil clean ideal of a ring is also nil clean ideal.
 \end{Thm}
 \begin{Thm}\label{Finite1}
 If $\{R_k\}_{k=1}^n$ be a finite collection of rings and $I_k$ an ideal of $R_k$, then $I=\prod_{k=1}^{n}I_k$ is nil clean ideal of $R=\prod_{k=1}^{n}R_k$ \textit{if and only if} $I_k$ is a nil clean ideal of $R_k$, for all $1\leq k\leq n$.
 \end{Thm}
 \begin{proof}
   Let $I_k$ is a nil clean ideal of $R_k$, for all $1\leq k\leq n$ and $x=(x_i)\in I$. There exists $m_i\in Nil(I_i)$ and $e_i\in Idem(I_i)$ such that $x_i=m_i+e_i$. Hence $(x_i)=(m_i)+(e_i)$, where $(m_i)\in Nil(I)$ and $(e_i)\in Idem(I)$. Conversely, for $x\in I_k$, consider $\overline{x}=(0,0,\cdot\cdot\cdot,0,x,0,\cdot\cdot\cdot,0)\in I$, where $k^{th}$ entry is $x$, so $\overline{x}=(e_i)+(m_i)$, where $(e_i)\in Idem(I)$ and $(m_i)\in Nil(I)$. Then $x=e_k+m_k$, where $e_k\in Idem(I_k)$ and $m_k\in Nil(I_k)$. Hence $I_k$ is nil clean ideal of $R_k$.
 \end{proof}
 The above Theorem \ref{Finite1} is not true for arbitrary collection of rings shown by given example.
 \begin{example}
   Consider the ring $R=\mathbb{Z}_2\times \mathbb{Z}_{2^2}\times \mathbb{Z}_{2^3}\times \cdot\cdot\cdot $, clearly for any $n\in \mathbb{N}$, $\mathbb{Z}_{2^n}$ is nil clean ring and hence ideal generated by $2$ in $\mathbb{Z}_{2^n}$ say $<2>_n$ is also nil clean ideal of $\mathbb{Z}_{2^n}$. But the ideal $I=<2>_1\times <2>_2\times <2>_3\times \cdot\cdot\cdot$ is not nil clean ideal of $R$ as $(2,2,2,\cdot\cdot\cdot)\in I$ can not be written as a sum of an idempotent and a nilpotent element of $R$.
 \end{example}
 Next we study the relationship between nil clean ideal of a given ring $R$ with nil clean ideal of upper triangular matrix ring $\mathbb{T}_n(R)$. Here given a matrix $X$, $X_{ij}$ denotes the $(i,j)$th entry of $X$.
 \begin{Lem}\label{D211}
   For $E, N\in \mathbb{T}_n(R)$ the following hold:
   \begin{enumerate}
     \item If $E^2=E$ then $(E_{ii})^2=E_{ii}$ for $1\leq i\leq n$.
     \item $N$ is nilpotent if and only if $N_{ii}$ is nilpotent for $1\leq i\leq n$.
   \end{enumerate}
 \end{Lem}
\begin{proof}
  See Lemma $2.1.1$ \cite{Diesl1}.
\end{proof}
\begin{Thm}\label{TT1}
  An ideal $I$ of $R$ is nil clean ideal \textit{if and only if} $\mathbb{T}_n(I)$ is nil clean ideal of $\mathbb{T}_n(R)$.
\end{Thm}
\begin{proof}
  Let $I$ be a nil clean ideal of $R$. Consider $A=(A_{ij})\in \mathbb{T}_n(I)$, then for $1\leq i\leq n$, $A_{ii}=E_{ii}+N_{ii}$, where $E_{ii}\in Idem(I)$ and $N_{ii}\in Nil(I)$, so $A=diag(E_{ii})+(B_{ij})$, where $B_{ii}=N_{ii}$ and $(B_{ij}\in \mathbb{T}_n(I))$. Clearly by Lemma \ref{D211} $diag(E_{ii})\in Idem(\mathbb{T}_n(I))$ and $B_{ij})\in Nil(\mathbb{T}_n(I)$. Hence $\mathbb{T}_n(I)$ is nil clean ideal of $\mathbb{T}_n(R)$. Conversely, for $x\in I$, consider $\overline{x}=diag(x,0,0,\cdot\cdot\cdot,0)\in \mathbb{T}_n(I)$ then $\overline{x}=(E_{ij})+(N_{ij})$, where $E_{ii}\in Idem(I)$ and $N_{ii}\in Nil(I)$. Hence $x=E_{11}+N_{11}$, as required.
  \end{proof}
Let $R$ be a commutative ring and $M$ be a $R$-module. Then the idealization of $R$ and $M$ is the ring $R(M)$ with underlying set $R\times M$ under coordinatewise addition and multiplication given by $(r,m)(r',m')=(rr', rm'+r'm)$, for all $r, r'\in R$ and $m, m' \in M$. It is obvious that if $I$ is an ideal of $R$ then for any submodule $N$ of $M$, $I(N)=\{(r,n)\, : \,r\in I \,\, ,\, \,n\in N \}$ is an ideal of $R(M)$. First we mention basic existing results about idempotents and units in $R(M)$ and study the nil clean ideals of the idealization $R(M)$ of $R$ and $R$-module $M$.
\begin{Lem}\label{RM}
  Let $R$ be a commutative ring and $R(M)$ be the idealization of $R$ and $R$-module $M$. Then the following hold:
  \begin{enumerate}
    \item  $(r,m) \in Idem(R(M))$ if and only if $r \in Idem(R)$ and $m =0$.
    \item  $(r,m) \in Nil(R(M))$ if and only if $r \in Nil(R)$.
  \end{enumerate}
\end{Lem}
\begin{proof}
  $(i)$ is obvious and $(ii)$ follows from the fact that $(r,m)^n=(r^n, nr^{n-1}m)$, for any $r\in R$ and $m\in M$.
\end{proof}
\begin{Prop}\label{RM1}
Let $R$ be a commutative ring and $R(M)$, the idealization of $R$ and $R$-module $M$. Then an ideal $I$ of $R$ is nil clean ideal of $R$ \textit{if and only if} $I(N)$ is nil clean ideal of $R(M)$, for any submodule $N$ of $M$.
\end{Prop}
\begin{proof}
  $(\Rightarrow)$ Consider $(x,m)\in I(N)$. For $x\in I$, $x=n+e$, where $n\in Nil(R)$ and $e\in Idem(R)$, so $(x,m)=(e,0)+(n,m)$, where $(e,0)\in Idem(R(M))$ and $(n,m)\in Nil(R(M))$, by Lemma \ref{RM}.\\

  $(\Leftarrow)$ Let $r\in I$, for $(r,0)\in I(N)$, $(r,0)=(e,0)+(n,n')$, where $(e,0)\in Idem(R(M))$, $(n,n')\in Nil(R(M))$ and $m,n'\in M$. Hence $r=e+n$, where $e\in Idem(R)$ and $n\in Nil(R)$ by Lemma \ref{RM}, as required.
\end{proof}
A Morita context denoted by $(R,S,M,N,\psi,\phi)$ consists of two rings $R$ and $S$, two bimodules $_RN_S$ and $_SM_R$ and a pair of bimodule homomorphisms (called pairings) $\psi:N\otimes _SM\rightarrow R$ and $\phi:M\otimes _RN\rightarrow S$, which satisfies the following associativity: $\psi(n\otimes m)n'=n\phi(m\otimes n')$ and $\phi (m\otimes n)m'=m\psi(n\otimes m')$, for any $m,\,m'\in M$ and $n,\,n'\in N$. These conditions ensure that the set of matrices $\left(
                                                                 \begin{array}{cc}
                                                                   r & n \\
                                                                   m & s \\
                                                                 \end{array}
                                                               \right)$, where $r\in R$, $s\in S$, $m\in M$ and $n\in N$ forms a ring denoted by $T$, called the ring of the context. For any subset $I$ of $T$, define $p_R(I)=\{a\in R\,:\,\left(
                                                                 \begin{array}{cc}
                                                                   a & x\\
                                                                   y & b \\
                                                                 \end{array}
                                                               \right)\in I \}$, $p_M(I)=\{y\in M\,:\,\left(
                                                                 \begin{array}{cc}
                                                                   a & x\\
                                                                   y & b \\
                                                                 \end{array}
                                                               \right)\in I \}$, $p_S(I)=\{b\in S\,:\,\left(
                                                                 \begin{array}{cc}
                                                                   a & x\\
                                                                   y & b \\
                                                                 \end{array}
                                                               \right)\in I \}$ and $p_N(I)=\{x\in N\,:\,\left(
                                                                 \begin{array}{cc}
                                                                   a & x\\
                                                                   y & b \\
                                                                 \end{array}
                                                               \right)\in I \}$
\begin{Lem}
  Let $T= \left(
                                                                 \begin{array}{cc}
                                                                   A & M \\
                                                                   N & B \\
                                                                 \end{array}
                                                               \right)$ be a morita context then $I$ is an ideal of $R$ \textit{if and only if} $I=\left(
                                                                 \begin{array}{cc}
                                                                   A_1 & M_1 \\
                                                                   N_1 & B_1 \\
                                                                 \end{array}
                                                               \right)$,
  where $A_1$ and $B_1$ are ideals of $A$ and $B$ respectively, $M_1\leq _AM_B$ and $N_1\leq _BN_A$ with $M_1N\subseteq A_1$, $N_1M\subseteq B_1$, $A_1M\subseteq M_1$, $B_1N\subseteq N_1$, $MN_1\subseteq A_1$, $NM_1\subseteq B_1$, $MB_1\subseteq M_1$ and $NA_1\subseteq N_1$. In this case $A_1=p_{A}(I)$, $B_1=p_B(I)$, $M_1=p_M(I)$ and $N_1=p_N(I)$.
\end{Lem}
\begin{proof}
  See Lemma $2.1$ \cite{Tang}.
\end{proof}
\begin{Thm}
  Let $R= \left(
                                                                 \begin{array}{cc}
                                                                   A & M \\
                                                                   N & B \\
                                                                 \end{array}
                                                               \right)$ be a morita context and $I= \left(
                                                                 \begin{array}{cc}
                                                                   A_1 & M_1 \\
                                                                   N_1 & B_1 \\
                                                                 \end{array}
                                                               \right)$ be a strongly nil clean ideal of $R$. Then $A_1$ and $A_2$ are strongly nil clean ideals of $A$ and $B$ respectively.
\end{Thm}
\begin{proof}
  The proof follows from Theorem \ref{main1}.
\end{proof}
A morita context $R= \left(
                                                                 \begin{array}{cc}
                                                                   A & M \\
                                                                   N & B \\
                                                                 \end{array}
                                                               \right)$ is called morita context of zero pairing if context products $MN=0$ and $NM=0$.

\begin{Thm}
  Let $R= \left(
                                                                 \begin{array}{cc}
                                                                   A & M \\
                                                                   N & B \\
                                                                 \end{array}
                                                               \right)$ be a morita context of zero pairing. Then $I= \left(
                                                                 \begin{array}{cc}
                                                                   A_1 & M_1 \\
                                                                   N_1 & B_1 \\
                                                                 \end{array}
                                                               \right)$ be a strongly nil clean ideal of $R$ \textit{if and only if} $A_1$ and $A_2$ are strongly nil clean ideals of $A$ and $B$ respectively.
\end{Thm}
\begin{proof}
 Let $A_1$ and $A_2$ be strongly nil clean ideals of $A$ and $B$ respectively. Let $x= \left(
                                                                 \begin{array}{cc}
                                                                   a & m \\
                                                                   n & b \\
                                                                 \end{array}
                                                               \right)\in I$,
 then there exist $k\in \mathbb{N}$, $e\in Idem(A_1)$, $p\in Nil(A_1)$, $f\in Idem(A_2)$ and $q\in Nil(A_2)$ such that $x= \left(
                                                                 \begin{array}{cc}
                                                                   e & 0 \\
                                                                   0 & f \\
                                                                 \end{array}
                                                               \right)+ \left(
                                                                 \begin{array}{cc}
                                                                   p & m \\
                                                                   n & q \\
                                                                 \end{array}
                                                               \right)$, $p^k=0$ and $q^k=0$. Now it is easy to see that $\left(
                                                                 \begin{array}{cc}
                                                                   p & m \\
                                                                   n & q \\
                                                                 \end{array}
                                                               \right)^k=\left(
                                                                 \begin{array}{cc}
                                                                   p^k & m_1 \\
                                                                   n_1 & q^k \\
                                                                 \end{array}
                                                               \right)=\left(
                                                                 \begin{array}{cc}
                                                                   0 & m_1 \\
                                                                   n_1 & 0 \\
                                                                 \end{array}
                                                               \right)$, where $m_1\in M$ and $n_1\in N$. So $\left(
                                                                 \begin{array}{cc}
                                                                   p & m \\
                                                                   n & q \\
                                                                 \end{array}
                                                               \right)^{2k}=\left(
                                                                 \begin{array}{cc}
                                                                   0 & 0 \\
                                                                   0 & 0 \\
                                                                 \end{array}
                                                               \right)$ and $\left(
                                                                 \begin{array}{cc}
                                                                   e & 0 \\
                                                                   0 & f \\
                                                                 \end{array}
                                                               \right)\in Idem(R)$. Hence $I$ is a nil clean ideal of $R$.
 \begin{Cor}
   Let $T$ be a $2\times2$ upper triangular matrix ring over $R$, then an ideal $\left(
                                                                 \begin{array}{cc}
                                                                   I & R \\
                                                                   0 & J \\
                                                                 \end{array}
                                                               \right)$ of $T$ is nil clean ideal \textit{if and only if} $I$ and $J$ are nil clean ideals of $R$.
 \end{Cor}
\end{proof}

\end{document}